\documentclass[a4paper,11pt]{article}
\usepackage{amssymb} 
\usepackage{amsmath} 
\usepackage{amsthm} 
\usepackage{a4wide}
\usepackage{hyperref}

\newtheorem{theorem}{Theorem}
\newtheorem{lemma}[theorem]{Lemma}
\newtheorem{claim}[theorem]{Claim}

\newcommand{\G}[2]{G_{#1,#2}}
\newcommand{\Exp}{\,\mathbb{E}}
\renewcommand{\Pr}{\,\mathbb{P}}
\newcommand{\eps}{\varepsilon}
\newcommand{\given}{ \; | \; }

\newcommand{\alphat}[1]{\alpha_{#1}}
\newcommand{\weakalphat}[1]{\hat{\alpha}_{#1}}
\newcommand{\alphatp}[2]{\alpha_{#1,#2}}
\newcommand{\weakalphatp}[2]{\hat{\alpha}_{#1,#2}}

\DeclareMathOperator{\Bin}{Bin}
\DeclareMathOperator{\avgdeg}{\overline{\deg}}

\title{Largest sparse subgraphs of random graphs}
\author{
Nikolaos Fountoulakis\\
University of Birmingham\\
{\tt n.fountoulakis@bham.ac.uk}
 \and
Ross J. Kang\\
Centrum Wiskunde \& Informatica\\
{\tt ross.kang@gmail.com}
 \and
Colin McDiarmid\\
University of Oxford\\
{\tt cmcd@stats.ox.ac.uk}}

\begin{document}

\maketitle

\begin{abstract}
  For the Erd{\H o}s-R\'enyi random graph $G_{n,p}$, we give a precise asymptotic formula for the size
  $\hat{\alpha}_{t}(G_{n,p})$ of a largest vertex subset in $G_{n,p}$ that induces a subgraph with
  average degree at most $t$, provided that $p = p(n)$ is not too small and $t = t(n)$ is not too large.
  In the case of fixed $t$ and $p$, we find that this value is asymptotically almost surely
  concentrated on at most two explicitly given points.  This generalises a result on the independence
  number of random graphs.  For both the upper and lower bounds, we rely on large deviations inequalities for
  the binomial distribution.
%
\end{abstract}

\section{Introduction}\label{sec:intro}

Given a graph $G = (V,E)$ and a non-negative number $t$, a vertex subset $S \subseteq V$ is {\em $t$-sparse} if the subgraph $G[S]$
induced by $S$ has average degree at most $t$.  The order of a largest such subset is called the {\em $t$-sparsity number} 
of $G$, denoted $\weakalphat{t}(G)$.  The $t$-sparsity number $\weakalphat{t}(G)$ is a natural generalisation of the
independence number $\alpha(G)$.  Recall that an independent set is a vertex subset of $G$ with no edges, i.e.~a $0$-sparse
set; thus the order $\alpha(G)$ of a largest independent set is just $\weakalphat{0}(G)$.  Note that $\weakalphat{t}(G)$ is non-decreasing in
terms of $t$.

We investigate the asymptotic behaviour of $\weakalphat{t}(\G{n}{p})$, where $\G{n}{p}$ is a random graph with vertex set $[n]=\{1,\dots,n\}$  and each edge is included independently at random with probability $p$.  We focus on fairly dense random graphs:~our main result holds when $p = p(n)$ satisfies $p \ge n^{-1/3+\eps}$ for some fixed $\eps > 0$ and $p$ bounded away from $1$.  
We say that a property holds {\em asymptotically almost surely (a.a.s.)} if it occurs with probability that tends to $1$ as $n \to \infty$.

For $t = 0$, that is, the independence number, the asymptotic behaviour in dense random graphs was described forty years ago by Matula~\cite{Mat70, Mat72, Mat76}, Grimmett and McDiarmid~\cite{GrMc75}, and Bollob\'as and Erd\H{o}s~\cite{BoEr76}.  For given $0 < p < 1$, define $b = 1/(1-p)$ and
\begin{align*}
\alpha_p(n) = 2\log_b n - 2 \log_b \log_b (n p) + 2 \log_b (e/2) + 1.
\end{align*}
It was shown that for any $\delta > 0$ a.a.s.~$\lfloor \alpha_p(n) - \delta \rfloor \le \alpha(\G{n}{p}) \le \lfloor \alpha_p(n) + \delta \rfloor$.  The main objective of this paper is to provide an analogue of this for $\weakalphat{t}(\G{n}{p})$.

Some previous estimates on $\weakalphat{t}(\G{n}{p})$ are implicit in the work of two of the authors.  In particular, for fixed $p$, it was observed using a first moment argument that for any $\eps > 0$, even if $t$ is a growing function of $n$, as long as $t = o(\ln (n p))$, we have $\weakalphat{t}(\G{n}{p}) \le (2 + \eps) \log_b (n p)$ a.a.s., cf.~\cite[Lemma~2.1]{KaMc07}.  It follows that $\weakalphat{t}(\G{n}{p})$ and $\alpha(\G{n}{p})$ share the same first-order term growth if $t = o(\ln (n p))$.
Furthermore, if $t = \omega(\ln (n p))$, then $(1 - \eps)t/p \le \weakalphat{t}(\G{n}{p}) \le (1 + \eps)t/p$, cf.~\cite[Lemma~2.2]{KaMc07}.
If $t = \Theta(\ln (n p))$, then the growth of the first-order term of $\weakalphat{t}(\G{n}{p})$ is a multiple of $\log_b (n p)$, and large deviation techniques were used to determine the factor (which depends on $p$ and $t$)~\cite{KaMc10}.  (With the exception of the precise factor at the threshold $t = \Theta(\ln (n p))$, these statements have been shown to remain valid for smaller values of $p$ as long as $p \gg 1/n$, cf.~\cite[Theorem~4.18]{Kan08}.)

In this work, we present a sharper description of $\weakalphat{t}(\G{n}{p})$, using a finer application of the above-mentioned methods.
However, we do not concern ourselves with the entire range of choices for the growth of $t$ as a function of $n$, as above.  To get our sharp formula with second- and third-order terms,
$p=p(n)$ must not tend to 0 too quickly, and $t=t(n)$ must not grow too quickly.
For $0 < p < 1$, define $b = 1/(1-p)$ and
\begin{align}\label{eqn:weakalphatp}
\weakalphatp{t}{p}(n) = 2\log_b n + (t - 2)\log_b \log_b (n p) - t \log_b t + t \log_b (2 b p e) + 2 \log_b (e/2) + 1.
\end{align}
Observe that $\weakalphatp{0}{p}(n) = \alpha_p(n)$ (under the convention that $0\ln0=0$) and also $\weakalphatp{t}{p}(n) = \alpha_p(n) + t\log_b((2bpe/t)\log_b(np))$. 
We prove the following.
\begin{theorem}\label{thm:dense,weak}
Let $0 < p = p(n) < 1$ be such that $p$ is bounded away from $1$ and $p > n^{-1/3+\eps}$, for some positive $\eps < 1/3$.  
Suppose $t = t(n) \ge 0$  and $\delta = \delta(n) > 0$ satisfy $t = o(\ln n / \ln \ln n)$ and $t^2 \ln \ln n/\ln n = o(p\delta)$.
Let $\weakalphatp{t}{p}(n)$ be as defined in~\eqref{eqn:weakalphatp}.
Then $\left\lfloor \weakalphatp{t}{p}(n) - \delta \right\rfloor \le \weakalphat{t}(\G{n}{p}) \le 
\left\lfloor \weakalphatp{t}{p}(n) + \delta \right\rfloor$ a.a.s.
\end{theorem}
\noindent
We see then that $\weakalphat{t}(\G{n}{p})$ is concentrated around $\weakalphatp{t}{p}(n)$ in an interval of width approximately $t^2 \ln \ln n/(p \ln n)$.  Thus, if $t^2 = o(p \ln n/\ln \ln n)$, then we have \emph{two-point concentration} (or \emph{focusing}), on explicit values.

Let us mention another related generalisation of the independence number.
Given a graph $G = (V,E)$ and a non-negative integer $t$, a vertex subset $S \subseteq V$ is {\em $t$-dependent} (or {\em $t$-stable}) if the subgraph $G[S]$ induced by $S$ has maximum degree at most $t$.  The order of a largest such subset is called the {\em $t$-dependence} (or {\em $t$-stability}) number of $G$, denoted $\alphat{t}(G)$.  Easily, $\alphat{t}(G) \le \weakalphat{t}(G)$.
 In~\cite{FKM10}, we considered $\alphat{t}(\G{n}{p})$, with our attention restricted to fixed $p$ and fixed $t$,
  in order to apply analytic techniques to the generating function of degree sequences on $k$ vertices
  and maximum degree at most $t$. For $0 < p < 1$, define
\begin{align*}
\alphatp{t}{p}(n) = 2\log_b n + (t - 2)\log_b \log_b (n p) + \log_b (t^t/t!^2) + t \log_b (2 b p/e) + 2 \log_b (e/2) + 1.
\end{align*}
 We showed in~\cite{FKM10} that for any fixed $\delta > 0$, $\left\lfloor \alphatp{t}{p}(n) - \delta \right\rfloor \le \alphat{t}(\G{n}{p}) \le \left\lfloor \alphatp{t}{p}(n) + \delta \right\rfloor$ a.a.s.
Note that in this setting the difference between the $t$-sparsity and the $t$-dependence numbers of $\G{n}{p}$ is essentially $\weakalphatp{t}{p}(n) - \alphatp{t}{p}(n) = 2 \log_b (t! e^t / t^t)$.  By Stirling's approximation for $t!$ (cf.~\cite{Bol01}), we have that $\weakalphatp{t}{p}(n) - \alphatp{t}{p}(n) \sim \log_b (2 \pi t)$ as $t \to \infty$.

We also comment here that, even if $t$ is fixed, the property of $t$-sparsity is not hereditary, i.e.~$t$-sparsity is not closed under vertex-deletion.  Hence the general asymptotic results of Bollob\'as and Thomason~\cite{BoTh00} (developed in a long line of research that can be traced back to early results of Alekseev~\cite{Ale82}, cf.~also~\cite{Ale92}), for partitions of random graphs according to a fixed hereditary property, are not applicable here.
In our previous studies~\cite{FKM10,KaMc10}, it was useful that $t$-dependence is hereditary for fixed $t$.  Unfortunately, this is not the case for $t$-sparsity.

As will become apparent, challenges arise in the second moment computations.  We have split this into several parts, according to the degree of overlap between two $k$-subsets of $[n]$.  Furthermore, in each part we must carefully account for the number of edges which are, say, within one of the $k$-subsets but not the other, or strictly contained in the overlap, and so on.  This careful accounting makes use of large deviations bounds for the binomial distribution.

The term ``sparse'' may take on a number of different meanings in graph theoretic or algorithmic research.  Instead of bounding average degree, one could instead bound for example the degeneracy (i.e.~the maximum over all subgraphs of the minimum degree) or the maximum average degree.   The counterparts of $t$-sparsity for these alternative versions of ``sparse'' are certainly of interest, but we do not pursue them here.  We remark only that the counterpart for the former example is bounded below by $\alphat{t}$, while for the latter example it is necessarily bounded between $\alphat{t}$ and $\weakalphat{t}$.

It is worth noting that the algorithmic complexity of computing the $t$-sparsity of a graph --- for the special cases of $t$ fixed or $t$ parameterised in terms of the order of the target set --- was recently studied by Bourgeois et al.~\cite{BGLMPP12} and, perhaps unsurprisingly, NP-hardness was shown to hold even in the restricted case of bipartite graphs.

Our paper is organised as follows.
In Section~\ref{sec:largedeviations}, we outline the large deviations results that we employ.
In Section~\ref{sec:Expectation}, we perform first moment calculations to obtain Lemma~\ref{lem:ExpectedAll}; this lemma implies the upper bound in Theorem~\ref{thm:dense,weak}.
In Section~\ref{sec:secondmoment}, we give a second moment calculation (Lemma~\ref{lem:dense,weak,lower}) which implies the lower bound in Theorem~\ref{thm:dense,weak}.


\section{Large deviations}\label{sec:largedeviations}

In this section, we state the large deviations techniques used to precisely describe the average degree of a $k$-set (a vertex subset of order $k$) in $\G{n}{p}$.
For background into large deviations, consult Dembo and Zeitouni~\cite{DeZe98}; we borrow some notation from this reference.  Given $0 < p < 1$, we let $q = 1 - p$ throughout.  Also, let
\[
\Lambda^*(x) = \left\{ \begin{array}{ll}
\displaystyle x \ln \frac{x}{p} + (1 - x) \ln \frac{1 - x}{q} & \mbox{for $x\in [0, 1]$}\\
\infty & \mbox{otherwise}
\end{array} \right.
\]
(where $\Lambda^*(0) = \ln (1/q)$ and $\Lambda^*(1) = \ln (1/p)$).
This is the Fenchel-Legendre transform of the logarithmic moment generating function associated with the Bernoulli distribution with probability $p$ (cf.~Exercise 2.2.23(b) of~\cite{DeZe98}).  Some easy calculus verifies that $\Lambda^*(x)$ has a global minimum of $0$ at $x = p$, is strictly decreasing on $[0, p)$ and strictly increasing on $(p,1]$.  

In the next lemma --- a large deviations result for the binomial distribution --- the upper bound follows easily from a strong version of Chernoff's bound, e.g.~(2.4) in~\cite{JLR00}, while the lower bound is implied by a sharp form of Stirling's formula, e.g.~(1.4) of~\cite{Bol01}: see the appendix of~\cite{KaMc10} for an explicit proof (when $r$ is integral).

\begin{lemma} \label{lem.bindev}
There is a constant $\delta>0$ such that the following holds.  Let $0<p<1$, let $N$ be a positive integer,
and let $X \in \Bin(N,p)$.  Then, for each  $ 1 \le r \le N-1$ such that $r\leq Np$,
\[
\delta \cdot\max\left\{r^{-1/2},(N-r)^{-1/2}\right\} \cdot \exp(-N\Lambda^*( r/N )) \le \Pr( X \le r) \le \exp(-N\Lambda^*( r/N )).
\]
\end{lemma}
Lemma~\ref{lem.bindev} immediately yields the following estimate on the probability that a given set of size $k$ is $t$-dependent. 
For a graph $G$, we let $\avgdeg(G)$ denote the average degree of $G$.
\begin{lemma}\label{lem:A_n}
Suppose $0 < p = p(n) < 1$ and suppose that $t = t(n) \geq 1$ and the positive integer $k = k(n)$ satisfy that 
$t \le p(k - 1)$. Then
\begin{enumerate}
\item\label{lem:A_n,i} $\displaystyle \Pr(\avgdeg(\G{k}{p}) \le t) \le \exp\left(-\binom{k}{2} \Lambda^*\left( \frac{t}{k - 1} \right) \right)$; and
\item\label{lem:A_n,ii} $\displaystyle \Pr(\avgdeg(\G{k}{p}) \le t) \ge \exp\left(-\binom{k}{2} \Lambda^*\left( \frac{t}{k - 1} \right) - \frac12 \ln k +O(1) \right)$.
\end{enumerate}
\end{lemma}

For the second moment estimation, we will make use of the following asymptotic calculations, the proofs of which are postponed to the appendix.

\begin{lemma}\label{lem:Lambda*}
Suppose $0 < p = p(n) < 1$ and suppose the non-negative number $t = t(n)$ and positive integer $k = k(n)$ satisfy that $t = o(p(k-1))$.  For any $\eps = \eps(n)$ with $|\eps| \le 1$,
\begin{align*}
\Lambda^*\left( \frac{(1+\eps)t}{k-1} \right)
= \Lambda^*\left(\frac{t}{k-1}\right)
 - (1 + o(1))\frac{\eps t}{k} \ln \frac{p k}{t}.
\end{align*}
\end{lemma}

\begin{lemma} \label{lem:LambdaApprox} 
Suppose $0 < p = p(n) < 1$ and that $x = x(n) =o(p)$. Then 
$$\Lambda^* (x) = \ln b \left( 1 + o(1) \right).$$
\end{lemma}

\noindent
We remark that we will throughout make implicit use of the fact that,  for $0 < x < 1$, $-x/(1-x) < \ln (1-x) < -x$.


\section{An expectation calculation for the upper bound}\label{sec:Expectation}
In this section, we consider the expected number of $t$-sparse $k$-sets.  Note that the range of valid values for $p$ in the following lemma is not as restrictive as for Theorem~\ref{thm:dense,weak}, and that the conditions for $t$ and $\delta$ are accordingly more general.

\begin{lemma} \label{lem:ExpectedAll}
Let $0 < p = p(n) < 1$ be such that $n p \to \infty$ as $n\to \infty$ and $p$ is bounded away from $1$. 
Suppose $t = t(n) \ge 0$ and $\delta = \delta(n) > 0$ satisfy $t = o(\ln (n p) / \ln \ln (n p))$ and $t^2 \log_b \ln (n p)/\ln (n p) = o(\delta)$.
Let $\weakalphatp{t}{p}(n)$ be as defined in~\eqref{eqn:weakalphatp}.
Let $k^+ = \lceil \weakalphatp{t}{p}(n) + \delta \rceil$ and $k^- = \lfloor \weakalphatp{t}{p}(n) - \delta \rfloor$ and let $\mathcal{S}_{n,t,k^+}$ and $\mathcal{S}_{n,t,k^-}$ be the collections of $t$-sparse $k^+$-sets and $k^-$-sets, respectively.
Then
\begin{align*}
  \Exp(|\mathcal{S}_{n,t,k^-}|)
 & \ge \exp\left((1+o(1))\delta\ln (n p)\right)\text{ and }\\
  \Exp(|\mathcal{S}_{n,t,k^+}|)
 & \le \exp\left(-(1+o(1))\delta\ln (n p)\right).
\end{align*}
\end{lemma}

\begin{proof}
Note that $\ln b = (1 + o(1))p$ if $p \to 0$ as $n\to \infty$.  For almost the entire proof, the calculations are carried out in terms of $k$, instead of $k^+$ or $k^-$.

By Lemma~\ref{lem:A_n},
\begin{align*}
\Exp(|\mathcal{S}_{n,t,k}|)
& = \binom{n}{k}\exp\left(-\binom{k}{2} \Lambda^*\left( \frac{t}{k - 1} \right) + O(\ln k)\right) \\
& = \left( \frac{en}{k} \right)^k \exp\left(-\left( \frac{k-1}{2} \right) \Lambda^*\left( \frac{t}{k - 1} \right) + O\left(\frac{\ln k}{k}\right)\right)^k\\
& = \exp\left(1 + \ln n - \ln k -\left( \frac{k-1}{2} \right) \Lambda^*\left( \frac{t}{k - 1} \right) + O\left(\frac{\ln k}{k}\right) \right)^k;
\end{align*}
therefore,
\begin{align}
\frac{2 \ln \Exp(|\mathcal{S}_{n,t,k}|)}{k}
& = 2 + 2 \ln n - 2 \ln k -(k-1) \Lambda^*\left( \frac{t}{k - 1} \right) + O\left(\frac{\ln k}{k}\right). \label{eqn:logAn/k,d}
\end{align}
Let us now expand one of the terms in~\eqref{eqn:logAn/k,d} using the formula for $\Lambda^*$:
\begin{align*}
(k-1) &\Lambda^*\left( \frac{t}{k - 1} \right)
 = t\ln \frac{t}{p(k - 1)} + (k - t - 1)\ln \left( \left( 1 - \frac{t}{k - 1} \right) \cdot \frac{1}{q} \right) \nonumber\\
& = t\ln t - t \ln (p(k - 1)) + (k - t - 1)\ln \left( 1 - \frac{t}{k - 1} \right) + (k - t - 1)\ln b.
\end{align*}
Since $|t/(k - 1)| < 1$ for $n$ large enough, we have by Taylor expansion that
\begin{align*}
            \ln \left( 1 - \frac{t}{k - 1} \right) & = - \frac{t}{k - 1} -\frac{t^2}{2(k - 1)^2} - \frac{t^3}{3(k - 1)^3} - \cdots, \text{ and}\nonumber\\
(k - t - 1) \ln \left( 1 - \frac{t}{k - 1} \right) & = -t + \frac{t^2}{2(k - 1)} + \frac{t^3}{6(k - 1)^2} + \cdots,
\end{align*}
giving that
\begin{align}
&\frac{2 \ln \Exp(|\mathcal{S}_{n,t,k}|)}{k} = \nonumber\\
& \ \ \ 2 + 2 \ln n - 2 \ln k - t\ln t + t \ln (p(k - 1)) + t - (k - t - 1)\ln b + O\left(\frac{t^2 + \ln k}{k}\right). \label{eqn:logAn/k2,d}
\end{align}
 Now, since $t \ge 0$, $n p\to \infty$ and $t \le \ln (n p)$ for $n$ large enough, it follows that $k \ge 2 \log_b (n p) - 2 \log_b \ln (n p)$ and
 \begin{align*}
 \ln n - \ln k & \le \ln n - \ln\left( 2 \log_b (n p) - 2 \log_b \ln (n p)  \right) \nonumber\\
    & \le \ln n - \ln \ln (n p) - \ln (2/\ln b) - \ln\left( 1 - \frac{\ln \ln (n p)}{\ln (n p)} \right)\nonumber\\
    & \le \ln n - \ln \ln (n p) - \ln (2/\ln b) + O\left(\frac{\ln \ln (n p)}{\ln (n p)}\right) 
 \end{align*}
 for $n$ large enough.
 Furthermore, for $n$ large enough,
 \begin{align*}
 t \ln (p(k-1)) & \le t \ln (p(2\log_b (n p) + t \log_b \ln (n p) )) \nonumber\\
 & \le t \ln \ln (n p) + t \ln (2p/\ln b) + t \ln \left( 1 + \frac{t \ln \ln (n p)}{2\ln (n p)} \right) \nonumber\\
 & \le t \ln \ln (n p) + t \ln (2p/\ln b) + \frac{t^2 \ln \ln (n p)}{\ln (n p)}. 
 \end{align*}
 Similarly, for $n$ large enough,
 \begin{align*}
 \ln n - \ln k 
 & \ge \ln n - \ln \ln (n p) - \ln (2/\ln b) + O\left(\frac{t \ln \ln (n p)}{\ln (n p)}\right)\text{ and }\\
 t \ln (p(k-1)) 
 & \ge t \ln \ln (n p) + t \ln (2p/\ln b) + O\left(\frac{t \ln \ln (n p)}{\ln (n p)}\right)
 \end{align*}
 so that
 \begin{align}
 \ln n - \ln k 
 & = \ln n - \ln \ln (n p) - \ln (2/\ln b) + O\left(\frac{t \ln \ln (n p)}{\ln (n p)}\right)\text{ and } \label{eqn:part1,d}\\
 t \ln (p(k-1)) 
 & = t \ln \ln (n p) + t \ln (2p/\ln b) + O\left(\frac{t^2 \ln \ln (n p)}{\ln (n p)}\right). \label{eqn:part2,d}
 \end{align}
Until here, our calculations did not depend on using $k^+$ or $k^-$, but now we have
 \begin{align*}
 & (k^- - t - 1)\ln b \le \\
 & 2 \ln n + (t-2)\ln \ln (n p) - (t-2)\ln \ln b - t\ln t + t\ln (2 p e) + 2 \ln (e/2) \pm \delta\ln b \\
 &\le (k^+ - t - 1)\ln b.
 \end{align*}
 Substituting the last inequalities together with~\eqref{eqn:part1,d} and~\eqref{eqn:part2,d} into~\eqref{eqn:logAn/k2,d}, we obtain, for $n$ large enough,
 \begin{align*}
\frac{2 \ln \Exp(|\mathcal{S}_{n,t,k^-}|)}{k^-}
 & \ge O\left(\frac{t^2 \ln \ln (n p)}{\ln (n p)}\right) + O\left(\frac{t^2 + \ln k}{k}\right) + \delta\ln b = 
       (1 + o(1))\delta\ln b\text{ and} \\
 \frac{2 \ln \Exp(|\mathcal{S}_{n,t,k^+}|)}{k^+}
 & \le O\left(\frac{t^2 \ln \ln (n p)}{\ln (n p)}\right) + O\left(\frac{t^2 + \ln k}{k}\right) - \delta\ln b = 
       -(1 + o(1))\delta\ln b,
 \end{align*}
 since $t^2 \ln \ln (n p)/\ln (n p) = o(\delta\ln b)$ and $k \ge \ln (n p)$. Now, substituting the expression $(1 + o(1)) 2 \log_b (n p)$ for $k^+$ or $k^-$ completes the proof.
\end{proof}

For illustration, let us consider the case of $p$ and $t$ fixed.
To satisfy the conditions in the above lemma we need $\delta \ln n / \ln\ln n \to \infty$ as $n \to \infty$.
So we may, for instance, set $\delta = (\ln\ln n)^2/\ln n$.
We find that the expected number of $t$-sparse sets of size $k^-$ tends to infinity.
The probability that there is a $t$-sparse set of size at least $k^+$ is at most $\Exp(|\mathcal{S}_{n,t,k^+}|) \to 0 \mbox{ as } n \to \infty$, and so $\weakalphat{t}(\G{n}{p}) \leq \lfloor \weakalphatp{t}{p}(n) + \delta \rfloor$ a.a.s.


\section{Second moment calculations for the lower bound}\label{sec:secondmoment}

\begin{lemma}\label{lem:dense,weak,lower}
Let $0 < p = p(n) < 1$ be such that $p$ is bounded away from 1 and $p > n^{-1/3 + \eps}$, 
for some positive $\eps < 1/3$.   
Suppose $t = t(n) \ge 0$ and $\delta = \delta(n) > 0$ satisfy $t = o(\ln n / \ln \ln n)$ and $t^2 \ln \ln n/\ln n = o(p\delta)$.
Let $\weakalphatp{t}{p}(n)$ be as defined in~\eqref{eqn:weakalphatp}.
If $k = k(n) = \lfloor \weakalphatp{t}{p}(n) - \delta \rfloor$, then
\begin{align*}
\Pr(\weakalphat{t}(\G{n}{p}) < k) = o(1). 
\end{align*}
\end{lemma}

\begin{proof}
Let $\mathcal{S}_{n,t,k}$ be the collection of $t$-sparse $k$-sets in $\G{n}{p}$.  By Lemma~\ref{lem:ExpectedAll},
\begin{align}
\Exp(|\mathcal{S}_{n,t,k}|) \ge \exp\left((1+o(1))\delta \ln (n p)\right). \label{eqn:exp}
\end{align}
We use Janson's Inequality (Theorem~2.18(ii) in~\cite{JLR00}):
\begin{align}
\Pr(\weakalphat{t}(\G{n}{p}) < k) = \Pr(|\mathcal{S}_{n,t,k}| = 0) \le \exp \left( - \frac{ \Exp^2(|\mathcal{S}_{n,t,k}|)}{\Exp(|\mathcal{S}_{n,t,k}|) + \Delta }\right), \label{eqn:Janson}
\end{align}
where
\[
\Delta = \sum_{A, B \subseteq [n], 1 < |A \cap B| < k} \Pr(A, B \in \mathcal{S}_{n,t,k}).
\]
We will split this sum into three sums according to the size of $|A \cap B|$ which we denote by $\ell$. 
In particular, let $p(k,\ell)$ be the probability that two $k$-subsets of $[n]$ that
overlap on exactly $\ell$ vertices are both in $\mathcal{S}_{n,t,k}$. 
Thus, 
\[ 
\Delta = \sum_{\ell = 1}^{k-1} \binom{n}{k} \binom{k}{\ell} \binom{n - k}{k-\ell} p(k,\ell). 
\]
For $\ell\in\{1,\ldots, k-1\}$, let $f(\ell) = \binom{n}{k} \binom{k}{\ell} \binom{n - k}{k-\ell} p(k,\ell)$.
We set $ \lambda_1 = \eps k/2$ and $\lambda_2 = (1-\eps)k$. (In fact, we shall assume throughout our proof 
that $\eps < 1/4$; note that this assumption still implies the lemma.)
Now we write $\Delta = \Delta_1 + \Delta_2 + \Delta_3$ where 
the parameters $\lambda_1$ and $\lambda_2$ determine the ranges of the three sums into which we decompose $\Delta$:
\begin{align*} 
\Delta_1 &=
 \sum_{1\le \ell < \lambda_1} f(\ell ), \ \
 \Delta_2 = \sum_{\lambda_1 \leq \ell < \lambda_2 } f(\ell ), \ \text{ and } \ \Delta_3 =\sum_{\lambda_2 \le \ell < k} f(\ell).
\end{align*}
We will show that for $i\in\{1,2,3\}$ we have  
\[\Delta_i = o \left( \Exp^2(\mathcal{S}_{n,t,k}) \right).  \]
So then the result follows from~\eqref{eqn:Janson}. 

To bound $\Delta_i$ for each $i\in\{1,2,3\}$, we consider two arbitrary $k$-subsets $A$ and $B$ of $[n]$ that overlap on exactly $\ell$ vertices,
i.e.~$|A\cap B| = \ell$, and estimate $p(k,\ell)$ by conditioning on the set $E[A\cap B]$ of edges induced by $A\cap B$.  In each of the
three regimes, we need slightly different techniques to estimate $p(k,\ell)$.


\subsubsection*{Bounding $\Delta_1$}

To bound $\Delta_1$,  
we write
\begin{align*}
p(k,\ell) = \Pr(A,B \in \mathcal{S}_{n,t,k}) = \Pr( A \in \mathcal{S}_{n,t,k} \given B \in \mathcal{S}_{n,t,k}) \cdot \Pr(B \in \mathcal{S}_{n,t,k}).
\end{align*}
The property of having average degree at most $t$ is monotone decreasing, so 
the conditional probability that $A \in \mathcal{S}_{n,t,k}$ is maximised when $E[A\cap B] = \emptyset$. 
Thus
\begin{align*}
\Pr( A \in \mathcal{S}_{n,t,k} \given B \in \mathcal{S}_{n,t,k})
& \le \Pr( A \in \mathcal{S}_{n,t,k} \given E[A\cap B] = \emptyset) \\
& \le \frac{\Pr(A \in \mathcal{S}_{n,t,k})}{\Pr(E[A\cap B] = \emptyset)} = 
b^{\binom{\ell}{2}} \Pr(A \in \mathcal{S}_{n,t,k})
\end{align*}
implying that $p(k,\ell) \le b^{\binom{\ell}{2}}\Pr^2(A \in \mathcal{S}_{n,t,k})$.  

We have though that for $n$ large enough
\begin{align*}
\frac{\binom{k}{\ell}~\binom{n-k}{k-\ell}}{\binom{n}{k}} \le 2\frac{\binom{k}{\ell} \cdot n^{k-\ell} / (k-\ell )!}{{n^k/k!}} = 
2 \left[\binom{k}{\ell} \right]^2~ \frac{\ell!}{n^\ell}.
\end{align*}
Thus
\begin{align*}
\Delta_1 & \leq 
\left(\binom{n}{k}\Pr(A \in \mathcal{S}_{n,t,k})\right)^2 \
\left( 2 \sum_{2 \leq \ell < \lambda_1} \left[\binom{k}{\ell} \right]^2~ \frac{\ell!}{n^\ell}~ b^{\binom{\ell}{2}} \right).
\end{align*}
We set 
\begin{align*}
s_{\ell}:= \left[\binom{k}{\ell} \right]^2~ \frac{\ell!}{n^\ell}~ b^{\binom{\ell}{2}}.
\end{align*}
Thus we write 
\begin{align*}
\Delta_1 \leq 2 \cdot \Exp^2(\mathcal{S}_{n,t,k}) 
\sum_{2\leq \ell < \lambda_1}  s_{\ell}.
\end{align*}
We will show that this sum is $o(1)$.

The following claim regards the monotonicity of $\{s_{\ell}\}$ for $\ell$ in the range of interest.  
\begin{claim} If $n$ is large enough, then for any $2\leq \ell < \lambda_1$ we have  $s_{\ell+1}/s_\ell < 1/2$. 
\end{claim}
\begin{proof}
We have 
\begin{align*}
\frac{s_{\ell+1}}{s_\ell} = \frac{(k-\ell)^2}{\ell +1}~\frac{b^{\ell}}{n} \leq \frac{k^2}{n} b^{\lambda_1} = 
O \left( \frac{n^{\eps}\log^2 n}{np^2}\right),
\end{align*}
as $b^{\lambda_1} = O(n^{\eps})$. 
But as $p \geq n^{-1/2 + \eps}$, we have $np^{2} \geq n^{2\eps}$ and, therefore, $s_{\ell+ 1}/s_\ell < 1/2$, for large enough 
$n$.  
\end{proof}
Thus the sum $\sum_{\ell < \lambda_1} s_{\ell}$ is essentially determined by its first term $s_2$:
\begin{align*}
\sum_{\ell < \lambda_1} s_{\ell} & \leq 2 s_2.
\end{align*} 
But we have 
\[
s_2 = O \left( \frac{k^4}{n^2} \right) = O \left( \frac{\log^4 n}{n^2p^4} \right) = O \left( n^2 \frac{\log^4 n}{(np)^4} \right) =o(1),
\]
if $p\geq n^{-1/2+\eps}$.


\subsubsection*{Bounding $\Delta_2$}

The bound on $\Delta_2= \sum_{\lambda_1 \leq \ell < \lambda_2 } f(\ell )$ involves a
more thorough consideration of the number of edges in the overlap between the sets $A$ and $B$. 

Let us fix some integer $\ell$ such that $\lambda_1 \leq \ell < \lambda_2$. We will show that 
$f(\ell )/\Exp^2(|\mathcal{S}_{n,t,k}|) = o(1/k)$.
With $A, B$ being two sets of vertices, each having size $k$, that overlap on $\ell$ vertices, we have 
\begin{equation} \label{eq:MasterRatio} 
\frac{f( \ell )}{\Exp^2(|\mathcal{S}_{n,t,k}|)} = \frac{ \binom{n-k}{k-\ell}~\binom{k}{\ell}}{\binom{n}{k}}
~\frac{\Pr( A, B \in \mathcal{S}_{n,t,k})}{\Pr^2( A \in \mathcal{S}_{n,t,k})}. 
\end{equation}
 The first ratio on the right-hand side can be bounded for $n$ sufficiently large as follows: 
\begin{align}\label{eqn:binom}
\frac{\binom{k}{\ell} \binom{n-k}{k-\ell}}{\binom{n}{k}} \le 2^{k+1}~\frac{n^{k-\ell} / (k-\ell)!}{n^k / k!} \leq 
2^{k+1}\binom{k}{\ell}~\frac{\ell!}{n^{\ell}} \leq 2^{2k+1} \left( \frac{k}{n} \right)^{\ell}.
\end{align}

We now give estimates on $\Pr( A, B \in \mathcal{S}_{n,t,k})$ as well as on $\Pr( A \in \mathcal{S}_{n,t,k})$. 
For each set $A$ of vertices, let $E[A]$ denote the set of edges with both their endvertices in $A$, and let $e(A)= |E[A]|$.  
Also, let $e'(A, B) = e(A) - e(A \cap  B)$, the number of edges in $E[A] \setminus E[A \cap B]$.
Setting $I=A\cap B$, we have 
\begin{equation*}
\Pr( A, B \in \mathcal{S}_{n,t,k} ) \leq \Pr( e(I) \leq kt/2 ) \cdot \Pr^2( e'(A, B) \leq kt/2 ). 
\end{equation*}
We will bound the two probabilities on the right-hand side of the above inequality using Lemma~\ref{lem.bindev}.  
 As $e(I) \in \Bin \left(\binom{\ell}{2},p\right)$ and $e'( A , B) \in \Bin \left( \binom{k}{2} - \binom{\ell}{2} ,p\right)$, 
with $x_{I}= kt/(\ell (\ell - 1))$ and $x_{A, B} = kt/(k(k-1) - \ell(\ell-1))$ we have  
\begin{align*}
&\Pr( e(I) \leq kt/2 ) = \exp \left(- \binom{\ell}{2}\Lambda^* \left( x_I\right) + O \left( \ln k \right) \right) \\
&\Pr( e'(A, B) \leq kt/2 ) = \exp \left(- \left(\binom{k}{2} - \binom{\ell}{2}\right) \Lambda^* 
\left( x_{A, B}\right) + O \left( \ln k \right) \right).
\end{align*}
Now, note that both $x_{I}$ and $x_{A, B}$ are $o(p)$. This holds since $x_I, x_{A, B}=O(t/k)$ and 
$k= \Theta ( \ln n/p)$ and $t = o( \ln n/ \ln \ln n)$. But now we can apply Lemma~\ref{lem:LambdaApprox} to obtain
\begin{equation} \label{eq:LargeDevsNum}
\begin{split}
& \Pr( e(I) \leq kt/2 ) \cdot \Pr^2( e'(A, B) \leq kt/2 ) \\
& = \exp \left( -  \binom{\ell}{2}\ln b (1+o(1)) - 2\left(\binom{k}{2} - \binom{\ell}{2}\right) \ln b (1+o(1)) + O   \left( \ln k \right)  \right) \\
&= \exp \left(   \binom{\ell}{2}\ln b - 2\binom{k}{2}\ln b  + o(k^2 p) \right).
\end{split}
\end{equation}
Similarly, 
\begin{equation} \label{eq:LargeDevsDenom} 
\Pr( A \in \mathcal{S}_{n,t,k} ) = \exp \left( - \binom{k}{2}\ln b + o(k^2 p) \right). 
\end{equation}
Hence the estimates in (\ref{eq:LargeDevsNum}) and (\ref{eq:LargeDevsDenom}) yield
\begin{align*} 
\frac{\Pr( A, B \in \mathcal{S}_{n,t,k} )}{\Pr^2( A \in \mathcal{S}_{n,t,k} )} = \exp \left( \binom{\ell}{2}\ln b 
+ o(k^2 p)\right). 
\end{align*} 
Now, combining the above together with (\ref{eqn:binom}) and the right-hand side of (\ref{eq:MasterRatio}), we obtain
\begin{equation} \label{eq:MasterRatioBound} 
\begin{split}
\frac{f(\ell)}{\Exp^2(|\mathcal{S}_{n,t,k}|)} & = \exp \left( - \ell \ln n + \ell \ln k + \binom{\ell}{2} \ln b + o( k^2 p)\right) \\  
&= \exp \left( - \ell \left( \ln n -  \ln k - \frac{\ell \ln b}{2} + o( k p) \right) \right). 
\end{split}
\end{equation}
We will show that $\ln n -  \ln k - \ell \ln b/2 \rightarrow \infty$ as $n \rightarrow \infty$, for any $\lambda_1 \leq \ell < \lambda_2$. 
Recall that $k = (2+o(1))\log_b (np)$. Thus $\ln n - \ln k = 
\ln \left( n \ln b/(2 \ln (np))\right) +o(1) \geq \ln \left( n p\right) +O(\ln \ln n)$ as $\ln b \geq p$. 
Also, as $\ell < (1-\eps ) k$, we have 
$\ell \ln b /2 < (1- \eps ) \ln (np) (1+o(1))$. Therefore 
\[
\ln n -  \ln k - \frac{\ell \ln b}{2} > \eps \ln (np) + o(\ln n).
\]

These two bounds substituted into (\ref{eq:MasterRatioBound}) now imply that 
\begin{equation} \label{eq:MasterRatioBoundFinal} 
\frac{f(\ell)}{\Exp^2( |\mathcal{S}_{n,t,k}|)} = \exp \left( - \Omega ( \ell \ln n)\right), 
\end{equation}
 uniformly for all $\lambda_1 \leq \ell < \lambda_2$. But since $\ell \geq \eps k/2$, this bound is $o(1/k)$ and therefore 
$\Delta_2 = o\left( \Exp^2(\mathcal{S}_{n,t,k}) \right)$.

\subsubsection*{Bounding $\Delta_3$}

Next, to bound $\Delta_3$, the aim here is also to show that for $\ell \geq \lambda_2$ we have 
\[
\frac{f(\ell)}{\Exp^2(|\mathcal{S}_{n,t,k}|) } = o\left( \frac{1}{k}\right).
\]
This is the portion of $\Delta$ that is the most difficult to control.  It is also the regime in which the condition $p \ge n^{-1/3+\eps}$ is required.  (We only required the weaker condition $p \ge n^{-1/2+\eps}$ to bound $\Delta_1$ and $\Delta_2$.)  In this regime, we need to separately treat two sub-regimes which are divided according to the edge count in the overlap.

Let us consider an arbitrary $\ell \geq \lambda_2$ and write
\begin{align*}
p(k,\ell) = \sum_{m = 0}^{\lfloor t k/2 \rfloor} p(k,\ell,m)
\end{align*}
where $p(k,\ell,m) = \Pr( A, B \in \mathcal{S}_{n,t,k} \land e( A\cap B ) = m )$.
(Note that $m \le t k/2$ or trivially both $A, B \notin \mathcal{S}_{n,t,k}$.)  We split this summation in two:
\begin{align}
p(k,\ell) = \sum_{m = 0}^{\mu} p(k,\ell,m) + \sum_{m = \mu + 1}^{\lfloor t k/2 \rfloor} p(k,\ell,m) =: p_1(k,\ell) + p_2(k,\ell),\label{eqn:edgecond}
\end{align}
with $\mu = \max\{0, \lfloor t k/2 - (k - \ell)(k + \ell - 1)\psi p/2 \rfloor \}$, where $\psi$ is the unique $0 < \psi = \psi(n) < 1$ such that $\Lambda^*(\psi p) = (1-\xi)\ln b$, for some fixed $0 < \xi < 1$ yet to be specified.

That $\psi$ exists is guaranteed by the fact that $\Lambda^*$ is strictly decreasing on $[0, p)$, $\Lambda^*(0) = \ln b$ and $\Lambda^*(p) = 0$.
  We now show that $\psi$ is bounded away from $0$.
  We have that $\psi$ satisfies
\begin{align*}
\xi
& = 1 - \frac{\Lambda^*(\psi p)}{\ln b}
= 1 - \left(\frac{\psi p}{\ln b} \ln \psi + \frac{1-\psi p}{\ln b}\ln\frac{1-\psi p}{q}\right) \\
& = \psi p + \frac{\psi p}{\ln b} \ln \frac{1}{\psi} - \frac{(1-\psi p)\ln (1-\psi p)}{\ln b} \\
& \leq \psi p + \frac{\psi p}{\ln b} \ln \frac{1}{\psi} + \frac{\psi p}{\ln b}
\end{align*}
  since $(1-x) \ln(1-x) \geq -x$ for $0<x<1$.  Thus, using also $p \leq \ln b$,
\begin{align*}
  \xi  \leq  \psi p + \psi \ln \frac{1}{\psi} + \psi \leq \psi (2 + \ln \psi)
\end{align*}
  But $x(2+\ln x) \to 0$ as $x \searrow 0$.  Hence there exists $\delta=\delta(\xi)>0$ such that
  $\psi \geq \delta$ uniformly over $p$.

Let us give a bound on $p_1(k,\ell)$.  We may assume that $(k - \ell)(k + \ell - 1)\psi p \le t k$, or else the sum is empty.
It will suffice to consider $E[A\cap B]$ alone.
Observe that $e ( A\cap B )$ is binomially distributed with parameters $\binom{\ell}2$ and $p$.  But $\ell \geq \lambda_2 = 
\Omega(\ln n/p)$, and so since $t = o(\ln n)$ it follows that $\mu \le t k/2 = o\left( p\binom{\ell}2 \right)$.  Thus, by Lemma~\ref{lem.bindev},
\begin{align*}
p_1(k,\ell)
& \le \Pr( e ( A\cap B ) \le \mu)
 \le \exp\left( -\binom{\ell}{2} \Lambda^*\left(\mu\left/\binom{\ell}{2}\right.\right) \right) \\
& = \exp\left( -\binom{\ell}{2} \Lambda^*\left(\frac{t k}{\ell(\ell - 1)} - \frac{(k - \ell)(k + \ell - 1) \psi p}{\ell(\ell - 1)} \right) \right).
\end{align*}
By Lemma~\ref{lem:Lambda*}, since $t k = o(p \ell(\ell - 1))$ and $0 \le (k - \ell)(k + \ell - 1)\psi p \le t k$,
\begin{align}
& p_1(k,\ell) 
 \le \exp\left( -\binom{\ell}{2} \left(\Lambda^*\left(\frac{t k}{\ell(\ell - 1)} \right) + (1 + o(1))\frac{(k - \ell)(k + \ell - 1) \psi p}{\ell(\ell - 1)}\ln \frac{p \ell(\ell - 1)}{t k}\right) \right) \nonumber\\
& = \exp\left( -\binom{\ell}{2} \Lambda^*\left(\frac{t k}{\ell(\ell - 1)} \right) 
 -(1 + o(1))\left(\frac{(k - \ell)(k + \ell - 1)}{2} \psi p\ln \frac{p \ell(\ell - 1)}{t k}\right) \right). \label{eqn:P1}
\end{align}

To estimate $p_2(k,\ell)$, we need a finer argument in which we also consider the sets $E[A]$ and $E[B]$ of edges induced by $A$ and $B$,
respectively.  In particular, let $X_1$ and $X_2$ denote $e'(A,B)$ (recall that this is $e(A)-e(A\cap B)$) and $e'(B,A)$, respectively.  Note that
$X_1$ and $X_2$ are binomially distributed with parameters $\ell(k - \ell) + \binom{k-\ell}2 = (k - \ell)(k + \ell - 1)/2$ and $p$.  Furthermore,
$X_1$ and $X_2$ and $e (A \cap B )$ are independent.  Therefore,
\begin{align*}
p_2(k,\ell)
& \le \Pr( e(A \cap B) \le t k/2) \cdot \Pr^2(X_1 \le t k/2 - \mu - 1).
\end{align*}
By Lemma~\ref{lem.bindev}, (since $t k/2 = o\left( p\binom{\ell}2 \right)$,)
\begin{align*}
\Pr(e(A \cap B ) \le t k/2)
& \le \exp\left( -\binom{\ell}{2} \Lambda^*\left(\frac{t k}{\ell(\ell - 1)} \right) \right)
\end{align*}
and (as $0 < \psi < 1$)
\begin{align*}
\Pr(X_1 \le t k/2 - \mu - 1)
& \le \Pr\left(X_1 \le \frac{(k - \ell)(k + \ell - 1)}{2} \psi p\right) \\
& \le \exp\left( -\frac{(k - \ell)(k + \ell - 1)}{2} \Lambda^*(\psi p) \right)\\
& = \exp\left( -\frac{(k - \ell)(k + \ell - 1)}{2} (1-\xi) \ln b \right).
\end{align*}
We conclude that
\begin{align}
p_2(k,\ell)
& \le \exp\left( -\binom{\ell}{2} \Lambda^*\left(\frac{t k}{\ell(\ell - 1)} \right) 
 -\frac{(k - \ell)(k + \ell - 1)}{2} (2-2\xi) \ln b \right).
\label{eqn:P2}
\end{align}
Comparing with~\eqref{eqn:P1}, since $t k = o(p \ell(\ell - 1))$ and $\psi = \Theta(1)$, we notice that $p_1(k,\ell)$ is asymptotically smaller than the above upper bound on $p_2(k,\ell)$.

Now, from $\ell \geq \lambda_2$ it follows that 
\begin{align*}
\ell(\ell - 1)
 \ge (k - \ell)(k + \ell - 1). 
\end{align*}
Indeed, $(k-\ell)(k+\ell - 1) \leq k^2 - \ell^2 \leq k^2- (1-\eps)^2 k^2 \leq 2\eps k^2$ and also for $n$ sufficiently large
$\ell (\ell -1) \geq (1-\eps)^2k^2 \geq (1-2\eps ) k^2$. As $\eps < 1/4$, the above inequality holds. 

Thus, since $t = o(p(k - 1))$, we obtain using Lemma~\ref{lem:Lambda*} that
\begin{align}
\binom{\ell}{2} \Lambda^*\left(\frac{t k}{\ell(\ell - 1)} \right)
& = \binom{\ell}{2} \Lambda^*\left( \left(1 + \frac{(k - \ell)(k + \ell - 1)}{\ell(\ell - 1)}\right)\frac{t}{k - 1}\right) \nonumber\\
& = \binom{\ell}{2} \Lambda^*\left(\frac{t}{k - 1}\right) - (1 + o(1)) \binom{\ell}{2} \frac{(k - \ell)(k + \ell - 1)t}{\ell(\ell - 1) k} \ln \frac{p k}{t} \nonumber\\
& = \binom{\ell}{2} \Lambda^*\left(\frac{t}{k - 1}\right) - (1 + o(1)) \frac{(k - \ell)(k + \ell - 1)}{2} p \frac{\ln (p k/t)}{p k/t} \nonumber\\
& = \binom{\ell}{2} \Lambda^*\left(\frac{t}{k - 1}\right) - o\left( \frac{(k - \ell)(k + \ell - 1)}{2} \ln b \right).
\label{eqn:ellapprox1}
\end{align}
Furthermore, since $\Lambda^*$ is strictly decreasing on $[0, p)$ and $\Lambda^*(0) = \ln b$,
\begin{align}
\binom{\ell}{2} \Lambda^*\left(\frac{t}{k - 1}\right)
& = \left( \binom{k}{2} - \frac{(k - \ell)(k + \ell - 1)}{2} \right) \Lambda^*\left(\frac{t}{k - 1}\right) \nonumber\\
& = \binom{k}{2} \Lambda^*\left(\frac{t}{k - 1}\right) - \frac{(k - \ell)(k + \ell - 1)}{2} \Lambda^*\left(\frac{t}{k - 1}\right) \nonumber\\
& \ge \binom{k}{2} \Lambda^*\left(\frac{t}{k - 1}\right) - \frac{(k - \ell)(k + \ell - 1)}{2} \ln b.
\label{eqn:ellapprox2}
\end{align}
Combining~\eqref{eqn:P2}--\eqref{eqn:ellapprox2}, we conclude that 
\begin{align} \label{eq:p2Bound}
p_2(k,\ell)
& \le \exp\left( -\binom{k}{2} \Lambda^*\left(\frac{t}{k - 1}\right)
-(1 + o(1)) \frac{(k - \ell)(k + \ell - 1)}{2} (1-2\xi) \ln b \right).
\end{align}
As remarked earlier, $p_1(k,\ell)$ is asymptotically smaller than the upper bound for $p_2(k,\ell)$.  
Hence it suffices to show that 
\[
\frac{\binom{n}{k}~\binom{k}{\ell}~\binom{n-k}{k-\ell}~p_2 (k,\ell )}{\Exp^2(|\mathcal{S}_{n,t,k}|)}  = o\left( \frac{1}{k}\right).
\]
Recall that with $A$ being a set of vertices of size $k$ we have
\[
\Exp(|\mathcal{S}_{n,t,k}|) = \binom{n}{k} \Pr(A\in \mathcal{S}_{n,t,k}) = \binom{n}{k}~\exp \left( - \binom{k}{2} \Lambda^*\left(\frac{t}{k - 1}\right) + O (\ln k) \right),
\]
where the last equality follows from Lemma~\ref{lem:A_n}. 
Thus, using (\ref{eq:p2Bound}), we have 
\begin{equation} 
\begin{split}
&\frac{\binom{n}{k}~\binom{k}{\ell}~\binom{n-k}{k-\ell}~p_2 (k,\ell )}{\Exp^2(|\mathcal{S}_{n,t,k}|)}   = \\
 & \frac{\binom{k}{\ell}~\binom{n-k}{k-\ell}}{\Exp(|\mathcal{S}_{n,t,k}|)}~ \exp\left( -(1 + o(1)) \frac{(k - \ell)(k + \ell - 1)}{2} (1-2\xi) \ln b + O(\ln k)\right).
\end{split}
\end{equation}
Now
\begin{align*}
{\binom{k}{\ell}~\binom{n-k}{k-\ell}} 
 \le (k n)^{k- \ell}.
\end{align*}
Thus using the lower bound on $\Exp(|\mathcal{S}_{n,t,k}|)$ given in (\ref{eqn:exp}) we obtain
\begin{equation} \label{eq:FinalRatio}
\begin{split}
&\ln \frac{\binom{n}{k}~\binom{k}{\ell}~\binom{n-k}{k-\ell}~p_2 (k,\ell)}{\Exp^2(|\mathcal{S}_{n,t,k}|)}   = \\
 & (k-\ell )\ln (kn) - (1+o(1))\delta \ln (np) -(1 + o(1)) \frac{(k - \ell)(k + \ell)}{2} (1-2\xi) \ln b + O(\ln k).
\end{split}
\end{equation}
Now, we have for $n$ sufficiently large
\begin{equation*}
\begin{split}
(k-\ell )&\ln (kn) -(1 + o(1)) \frac{(k - \ell)(k + \ell)}{2} (1-2\xi) \ln b \\
 & \leq (k-\ell) \left( \ln (nk) - \frac{(2-\eps)k}{2} (1-2\xi) \ln b \right) \\ 
&\leq (k-\ell) (\ln (nk) - (2 - 2\eps) \ln (n p)),  
\end{split}
\end{equation*}
where in the last inequality we used a choice of $\xi$ small enough as well as the fact that $k \ln b = (1+o(1))2\ln (np)$.
But also $k \leq (1+o(1)) 2\ln (np) /p$, as $\ln b \geq p$. This implies that 
$\ln k \leq \ln \ln (n p) - \ln p + O(1)$. Hence, for $n$ sufficiently large,
\begin{align*}
(k-\ell )&\ln (kn) -(1 + o(1)) \frac{(k - \ell)(k + \ell)}{2} (1-2\xi) \ln b \\
& \leq (k-\ell) (-\ln n + \ln \ln (n p) + O(1) - 3 \ln p + 2\eps \ln (n p) ) \\
& \le (k-\ell) (-\ln n - 3\ln n^{-1/3+\eps} + 3\eps \ln (n p)) \le 0,
\end{align*}
where we used the condition $p \ge n^{-1/3+\eps}$ in the second last inequality.
Substituting this into (\ref{eq:FinalRatio}), we obtain 
\begin{equation*} 
\begin{split}
\frac{\binom{n}{k}~\binom{k}{\ell}~\binom{n-k}{k-\ell}~p_2 (k,\ell )}{\Exp^2(|\mathcal{S}_{n,t,k}|)} \leq 
\exp\left( - (1+o(1))\delta \ln (np) + O(\ln k)\right).
\end{split}
\end{equation*}
But $\ln \ln n/\ln n = o(p\delta)$ and therefore $\ln \ln n/p = o(\delta \ln (np))$.
On the other hand, $\ln k = O\left(\ln ( \ln n/ p ) \right)$, which implies that $\ln k = o(\delta \ln (np))$. 
Therefore
\begin{equation*} 
\begin{split}
\frac{\binom{n}{k}~\binom{k}{\ell}~\binom{n-k}{k-\ell}~p_2 (k,\ell )}{\Exp^2(|\mathcal{S}_{n,t,k}|)} = o\left( \frac{1}{k} \right),
\end{split}
\end{equation*}
as required. 
\end{proof}

\bibliographystyle{abbrv}
\bibliography{avtdep}

\appendix

\section{Appendix}

\begin{proof}[Proof of Lemma~\ref{lem:Lambda*}]
We split the proof into two cases.
First, if $\eps = -1$, then
\begin{align*}
&\Lambda^*\left( \frac{t}{k-1} \right) - \Lambda^*\left( \frac{(1+\eps)t}{k-1} \right)
= \Lambda^*\left( \frac{t}{k-1} \right) - \Lambda^*(0) \\
& = \frac{t}{k-1}\ln\frac{t}{p(k-1)} + \left(1-\frac{t}{k-1}\right)\ln\frac{1-\frac{t}{k-1}}{q} - \ln\frac{1}{q}\\
& = \frac{t}{k-1}\ln\frac{t}{p(k-1)} + \left(1-\frac{t}{k-1}\right)\ln\left(1-\frac{t}{k-1}\right) - \frac{t}{k-1}\ln\frac{1}{q}\\
& = \frac{t}{k-1}\ln\frac{qt}{p(k-1)} - \frac{t}{k-1} + O\left(\frac{t^2}{k^2}\right)
 = -(1 + o(1))\frac{t}{k}\ln\frac{p k}{t}\\
& = (1 + o(1))\frac{\eps t}{k}\ln\frac{p k}{t}
\end{align*}
(where we used $t = o(k)$ and the Taylor expansion of $(1 - t/(k-1))\ln (1 - t/(k-1))$).
Otherwise, $-1 < \eps \le 1$ and
\begin{align*}
\Lambda^*\left( \frac{(1+\eps)t}{k-1} \right)
& = \left(\frac{(1+\eps)t}{k-1}\right) \ln \frac{(1+\eps)t}{p(k-1)} + \left(1-\frac{(1+\eps)t}{k-1}\right) \ln \frac{k-1-(1+\eps)t}{q(k-1)} \\
& = \frac{t}{k-1} \ln \frac{(1+\eps)t}{p(k-1)} + \frac{\eps t}{k-1} \ln \frac{(1+\eps)t}{p(k-1)} \\
& \ \ \ \ \ \ \ + \left(1 - \frac{t}{k-1}\right) \ln \frac{k-1-(1+\eps)t}{q(k-1)} - \frac{\eps t}{k-1} \ln \frac{k-1-(1+\eps)t}{q(k-1)} \\
& = \Lambda^*\left(\frac{t}{k-1}\right)
 + \frac{t}{k-1} \ln (1+\eps) + \left(1 - \frac{t}{k-1}\right) \ln \left( 1 - \frac{\eps t}{k-1-t} \right) \\
& \ \ \ \ \ \ \ + \frac{\eps t}{k-1} \ln \frac{q(1+\eps)t}{p(k-1-(1+\eps)t)}
\end{align*}
and the lemma follows by observing that, by Taylor expansion,
\begin{align*}
\frac{t}{k-1} \ln (1+\eps) + \left(1 - \frac{t}{k-1}\right) \ln \left( 1 - \frac{\eps t}{k-1-t} \right)
 = O\left(\frac{\eps^2 t}{k}\right)&\text{ and}\\
\frac{\eps t}{k-1} \ln \frac{q(1+\eps)t}{p(k-1-(1+\eps)t)}
 = -(1 + o(1))\frac{\eps t}{k} \ln \frac{p k}{t}&.
\end{align*}
\end{proof}

\begin{proof}[Proof of Lemma~\ref{lem:LambdaApprox}]
Since $(1-x) \ln (1-x)=O(x)$ as $x \to 0$,
\begin{equation*} 
\begin{split}
\Lambda^* (x) 
&= x \ln \left( \frac{x}{p} \right) + (1-x) \ln b + (1-x) \ln (1-x) \\
&= \ln b \left( 1 + \frac{x}{\ln b} \ln \left( \frac{x}{p} \right) + O \left( \frac{x}{\ln b} \right)  \right).
\end{split}
\end{equation*}
But $p = \Theta (\ln b)$ and $x = o(p)$,
and the lemma follows. 
\end{proof}

\end{document}